\documentclass[reqno,12 pt]{amsart}

\usepackage{amsfonts,amscd}
\usepackage{amsthm}
\usepackage{amssymb}
\usepackage{latexsym}
\usepackage{wasysym}
\usepackage{cite}
\usepackage{lineno}
\usepackage{stmaryrd}
\usepackage{url}
\usepackage{bm}
\usepackage{amsmath}
\usepackage{resizegather}
\usepackage{verbatim}
\usepackage{graphicx}
\usepackage{epsf}
\usepackage{enumerate}
\usepackage{geometry}
\usepackage{hyperref}
\hypersetup{breaklinks=true,
pagecolor=white,
colorlinks=true,
citecolor=blue,%
filecolor=black,%
linkcolor=blue,%
urlcolor=blue
}

\theoremstyle{plain} 
\newtheorem{thm}{Theorem}
\newtheorem{lem}[thm]{Lemma}

\theoremstyle{definition}
\newtheorem{remark}[thm]{Remark}
\newtheorem{expl}[thm]{Example}

\newcommand{\wt}{\widetilde}

\newcommand{\Id}{\textnormal{Id}}

\newcommand{\R}{\mathbb{R}}

\newcommand{\N}{\mathbb{N}}

\newcommand{\wh}{\widehat}

\numberwithin{equation}{section}
\newcommand{\bref}[1]{\textbf{\ref{#1}}} 
\newcommand{\beqref}[1]{\textbf{(\ref{#1})}} 

\date{July 28, 2026. }

\begin{document}

\title[Boundedness of infinite products of relaxed projections]{On the boundedness of infinite products of relaxed projections: perturbations resilience and dynamic string-averaging}


\author{Daniel Reem}
\address{The Center for Mathematics and Scientific Computation (CMSC), University of Haifa, Mt. Carmel, Haifa 3498838, Israel}
\email{dream@math.haifa.ac.il}

\author{Yair Censor}
\address{Department of Mathematics, University of Haifa, Mt. Carmel, Haifa 3498838, Israel}
\email{yair@math.haifa.ac.il}





\begin{abstract}
Very recently (2026), Bauschke and Tung extended from finite- to infinite-dimensional Hilbert spaces a result published by Meshulam in 1996 (following an earlier result of Aharoni-Duchet-Wajnryb from 1984) regarding the boundedness of infinite products of relaxed projections onto a finite family of closed affine subspaces. In the present note we extend in various ways the result of Bauschke and Tung by allowing certain perturbations and proving perturbation resilience, by considering a mixture of closed half-spaces and closed hyperplanes, and by using dynamic weighted sums of dynamic strings (of dynamic lengths) of relaxed projections in the iterative process. We also discuss the limitation to generalize the Bauschke-Tung result to arbitrary closed and convex sets by presenting a large family of counterexamples in which the associated control is not cyclic and not even almost cyclic. Along the way we establish a general theorem of independent interest regarding the uniform boundedness and uniform unboundedness of infinite products of nonexpansive mappings in a normed space setting. 
\end{abstract}

\maketitle

\section{Introduction}\label{sec:Intro}
\subsection{Background:} Following the celebrated results of von Neumann regarding the so-called method of alternating projections onto closed linear subspaces in a Hilbert space \cite[Lemma 22, p. 475]{vonNeumann1949jour}, \cite[Theorem 13.7, pp. 55--56]{vonNeumann1950book} (results which were published in 1949--1950, but were discovered and used already in 1933--1935), numerous results related to infinite products of projections onto various subsets in various settings have been published by many authors, until this very date; see, for instance, the following very partial list of references and the references therein: \cite{AmemiyaAndo1965jour,BauschkeBorwein1993jour,
BauschkeBorwein1996jour,BauschkeBorweinLewis1997inproc,BauschkeCombettes2017book,
BauschkeMatouskovaReich2004jour,BehlingBello-CruzSantos2021jour,Browder1958jour,Cegielski2012book,CheneyGoldstein1959jour,
CensorZaknoon2018jour,Deutsch1979incol,Deutsch1983incol,Deutsch2001book,DeutschHundal2008jour,
Dye1989jour-a,Dye1989jour-b,DyeKhamsiReich1991jour,DyeReich1991jour,EskandariMoslehian2026jour,
FranchettiLight1984jour,FranchettiLight1986jour,GubinPolyakRaik1967jour,GunturkThao2023jour,
Halperin1962jour,Hundal2004jour,KayalarWeinert1988jour,KopeckaMuller2014jour,
KopeckaPaszkiewicz2017jour,KopeckaReich2004jour, KopeckaReich2012jour,KosmolZhou1991jour,Nakano1953book,Oppenheim2018jour,
Paszkiewicz2012prep,Pinkus2015jour,Prager1960jour,PustylnikReichZaslavski2011jour,
PustylnikReichZaslavski2012jour,PustylnikReichZaslavski2012jour-b,Stiles1965a-jour,Stiles1965b-jour,Wiener1955jour}.

One of these results was published in 1984 by Aharoni, Duchet and Wajnryb \cite[the Theorem on page 134]{AharoniDuchetWajnryb1984jour}. In a nutshell, this interesting result says that given a finite number of (necessarily closed) affine subspaces in a Euclidean space, if one starts at any point in the space, projects it onto one (arbitrary) of the affine subspaces, then projects the resulting point onto another (arbitrary) affine subspace from the family, and continues this process indefinitely so that one obtains an infinite product of projections onto the given affine subspaces, then this product is bounded. The proof of this result was complicated, and a simplified and neat version of it was published in 1996 by Meshulam \cite[Theorem 1]{Meshulam1996jour}, who was also able to present explicit estimates on the bound in terms of the initial point and the family of affine subspaces. 

At the end of his paper, Meshulam noted that  after Aharoni saw his (Meshulam's) new proof of \cite[the Theorem on page 134]{AharoniDuchetWajnryb1984jour}, Aharoni observed that Meshulam's method of proof can be used in order to establish a more general result, in which the orthogonal projections onto the affine subspaces are replaced by relaxed orthogonal projections. This more general result was formulated without a proof as \cite[Theorem 2]{Meshulam1996jour}. 

The proofs of both \cite[the Theorem on page 134]{AharoniDuchetWajnryb1984jour} and \cite[Theorem 1]{Meshulam1996jour} were heavily based on the finite-dimensional setting, a fact that has hindered for about 30 years the publication of any infinite-dimensional generalization of their results. Very recently, the situation has changed, when Bauschke and Tung have been able to generalize in a clever way \cite[Theorem 2]{Meshulam1996jour} and published the following result (see the penultimate paragraph of Section \bref{sec:Preliminaries} for the definition of an innately regular family of closed affine subspaces  that is used in Theorem \bref{thm:BauschkeTung} below; note that, as explained there, when the subspaces in the family are finite-dimensional, then the family automatically becomes innately regular):

\begin{thm}\label{thm:BauschkeTung} {\bf (Bauschke-Tung \cite[Corollary 4.3]{BauschkeTung2026prep} (2026)} Suppose that $X$ is a real Hilbert space, $m\in\N$ and $\lambda\in (0,2)$. Suppose further that $\mathcal{A}:=(A_i)_{i=1}^m$ is a family of $m$ closed affine subspaces in $X$ which is innately regular. Then there is a universal positive constant $\mu_{\mathcal{A},\lambda}>0$, depending on $\lambda$ and on the family $\mathcal{A}$, which has the following property: for each $k\in\N\cup\{0\}$, each $\lambda_k\in [0,\lambda]$, each control $c:\N\cup\{0\}\to I:=\{1,2,\ldots,m\}$, and each $x^0\in X$, if 
\begin{equation}\label{eq:x^(k+1)}
x^{k+1}:=T_{k}(x^k),
\end{equation}
where $T_k:X\to X$ is defined by $T_k(x):=(1-\lambda_k)x+\lambda_k P_{A_{c(k)}}(x)$ for all $x\in X$ and $P_{A_{c(k)}}$ is the orthogonal projection operator onto the closed affine subspace $A_{c(k)}$ (in other words, $x^{k+1}$ is the relaxed projection of $x^k$ onto the closed affine subspace $A_{c(k)}$), then one has $\|x^k\|\leq \|x^0\|+\lambda \mu_{\mathcal{A},\lambda}$. Stated differently, if we let 
\begin{equation}\label{eq:R_A_lambda}
\mathcal{R}_{\mathcal{A},\lambda}:=\{S:X\to X\,|\, S=(1-\tau)\Id+\tau P_{A},\,\tau\in [0,\lambda],\,\textnormal{and}\,A\in \{A_i\,|\,i\in I\}\},
\end{equation}
where $\Id$ is the identity map, then there is a universal positive constant $\mu_{\mathcal{A},\lambda}>0$, depending on $\lambda$ and on the family $\mathcal{A}$, such that for all $k\in\N$, all $x^0\in X$ and all $k$ operators $S_{k-1},S_{k-2},\ldots, S_0$ from the set $\mathcal{R}_{\mathcal{A},\lambda}$, we have $\|S_{k-1}S_{k-2}\cdots S_0(x^0)\|\leq \|x^0\|+\lambda \mu_{\mathcal{A},\lambda}$.
\end{thm}
Note that since in general the sequences generated in Theorem \bref{thm:BauschkeTung} (and in its finite-dimensional predecessors \cite[the Theorem on page 134]{AharoniDuchetWajnryb1984jour}, \cite[Theorem 1, Theorem 2]{Meshulam1996jour}) are not Fej\'er monotone with respect to $\cap_{i=1}^m A_i$ (for instance, because it may happen that $\cap_{i=1}^m A_i=\emptyset$), it is not at all clear that these sequences must be bounded, and by a uniform bound as provided in Theorem \bref{thm:BauschkeTung} (see also \cite[The proof of Theorem 1]{Meshulam1996jour}). We note that a bit before publishing \cite{BauschkeTung2026prep}, Bauschke and Tung published another version of Theorem \bref{thm:BauschkeTung}, again in a Hilbert space setting, in which the closed affine subspaces were replaced by polyhedral sets (that is, by finite intersections of closed half-spaces): see \cite[Theorem 3.2]{BauschkeTung2025prep}, and see other parts of \cite{BauschkeTung2025prep} for interesting related examples and counterexamples. 

It is worth mentioning a few other related results. One of them is \cite[Theorem 1]{BaranyGoodmanPollack1991inbook}, which is a version of \cite[the Theorem on page 134]{AharoniDuchetWajnryb1984jour}, \cite[Theorem 1]{Meshulam1996jour} established by B\'ar\'any, Goodman and Pollack in which the setting is a possibly infinite family of non-parallel lines having the property that the family of pairwise intersections of these lines is bounded. Another result is \cite[Theorem 3]{BaranyGoodmanPollack1991inbook} in which \cite[Theorem 1]{BaranyGoodmanPollack1991inbook} was generalized to any Euclidean space of dimension 2 and above under some assumptions on the infinitely many (of the continuum cardinal) lines. A third result is of Stiles \cite[pp. 28--29]{Stiles1965a-jour}, in which one can find a counterexample in a non-Hilbert space setting with two closed linear subspaces where the obtained sequences are unbounded despite the fact that the subspaces intersect at the origin. See also \cite{AmaldiHauser2005jour,BauschkeTung2026prep-b,BlockLevin1970jour,Efron1964tech-report,MinskyPapert2017book,Ridgway1962phd} and the references therein for results related to the boundedness of the so-called perceptron algorithm and variants of it applied to infeasible and feasible systems of linear and convex inequalities.  

\subsection{Contributions:} In this note we extend in various ways Theorem \bref{thm:BauschkeTung} by allowing certain perturbations to appear in this theorem (see Theorems \bref{thm:BoundedPerturbations}--\bref{thm:SummablePerturbations}), by allowing a mixture of closed half-spaces and closed hyperplanes (see Theorem \bref{thm:Halfspaces}), and by allowing dynamic weighted strings of dynamic lengths of relaxes projections instead of the relaxed projections (see Theorem \bref{thm:BoundedString}). We also discuss the limitation on generalizing Theorem \bref{thm:BauschkeTung} to arbitrary closed and convex sets by presenting a large family of counterexamples in which the associated control is not cyclic and not even almost cyclic: see Examples \bref{ex:(C1C2)k(C_3)}--\bref{ex:Aleph}. The analysis of these examples is based on a general theorem of independent interest (Theorem \bref{thm:UniformInfty}) regarding the uniform boundedness and uniform unboundedness of infinite products of nonexpansive mappings in a normed space setting. 

\subsection{Paper layout:} The paper is organized as follows. In Section \bref{sec:Preliminaries} we present preliminaries necessary for the subsequent developments. Perturbations are introduced in Section \bref{sec:Perturbations} into the overall iterative process and are shown to retain the boundedness property. In Section \bref{sec:Halfspaces} we discuss closed half-spaces and in Section \bref{sec:Dynamic-string-averaging} the scenario is expanded from plain sequential projections to dynamic string-averaging projections. Finally, in Section \bref{sec:Limitations-on-the} we show, by means of several examples, that Theorem \bref{thm:BauschkeTung} cannot be extended without further assumptions to arbitrary nonempty, closed and convex subsets.

\section{Preliminaries}\label{sec:Preliminaries}
Unless otherwise stated, our setting is a real Hilbert space $X\neq\{0\}$ with an inner product $\langle\cdot,\cdot\rangle$ and an induced norm $\|\cdot\|$. If $C\subseteq X$ is a nonempty, closed and convex subset of $X$, then $P_C:X\to X$ denotes the orthogonal (or best approximation) projection operator onto $C$, whose existence and uniqueness are well known \cite[Theorem 3.16, p. 53]{BauschkeCombettes2017book}, \cite[Theorem 1.2.3, p. 18]{Cegielski2012book}. Given a finite set $J\neq\emptyset$, a weight function on $J$ is a function $w:J\to [0,1]$ which satisfies  $\sum_{j\in J}w(j)=1$. A control with respect to a nonempty index set $I$ is a function $c$ from the set $\N\cup\{0\}$ of nonnegative integers to $I$. 

A control $c$ is said to be cyclic with period $p\in \N$ if $c(k+p)=c(k)$ for all $k\in\N\cup\{0\}$; a control $c$ is said to be almost cyclic with an almost cycle $p\in \N$ if $I\subseteq \{c(k),c(k+1),\ldots,c(k+p-1)\}$ for all $k\in\N\cup\{0\}$; a control on which there is no restriction is said to be random or chaotic. Given a sequence $(y^k)_{k=0}^{\infty}$ in $X$ and $p\in\N$, we say that $(y^k)_{k=0}^{\infty}$ is periodic with period $p$ if $y^{k+p}=y^k$ for all $k\in \N\cup\{0\}$; we say that $(y^k)_{k=0}^{\infty}$ is almost-periodic with an almost-period $p$ if $\{y^k\,|\,k\in\N\cup\{0\}\}=\{y^t,y^{t+1},\ldots,y^{t+p-1}\}$ for all $t\in\N\cup\{0\}$, namely each block of $p$ consecutive elements of the sequence contains all the elements of the sequence.

$\Id:X\to X$ denotes the identity operator, namely $\Id(x)=x$ for all $x\in X$. Given an operator $T:X\to X$, we say that $T$ is affine whenever there are a point $p_T\in X$ and a linear operator $\wt{T}:X\to X$ such that $T=p_T+\wt{T}$. We say that $T$ is linear convex if $T(\alpha x+\beta y)=\alpha T(x)+\beta T(y)$ for all $x,y\in X$ and all $\alpha,\beta\in [0,1]$.  We say that $T$ is nonexpansive if it is 1-Lipschitz continuous, namely if $\|T(x)-T(y)\|\leq \|x-y\|$ for all $x$ and $y$ in $X$. We say that $T$ is firmly nonexpansive if $\|T(x)-T(y)\|^2\leq \|x-y\|^2-\|(x-T(x))-(y-T(y))\|^2$ for all $x,y\in X$. We say that $T$ is a relaxation of another operator $S:X\to X$ with a relaxation parameter $\lambda\in\R$, or that $T$ is the $\lambda$-relaxation of $S$, whenever $T=(1-\lambda)\Id+\lambda S$. It is well known that an orthogonal projection onto a nonempty, closed and convex subset is firmly nonexpansive \cite[Theorem 2.2.21, p. 76]{Cegielski2012book} and also that the $\lambda$-relaxation of a firmly nonexpansive operator is nonexpansive whenever $\lambda\in [0,2]$ (see \cite[Theorem 2.2.10, p. 70]{Cegielski2012book}). 

We say that $T$ is strongly nonexpansive if $T$ is nonexpansive and has the property that for all sequences $(x^k)_{k=0}^{\infty}$ and $(y^k)_{k=0}^{\infty}$ in $X$ which satisfy both $\sup\{\|x^k-y^k\|\,|\, k\in\N\cup\{0\}\}<\infty$ and $\lim_{k\to\infty}(\|x^k-y^k\|-\|T(x^k)-T(y^k)\|)=0$, one has $\lim_{k\to\infty}((x^k-y^k)-(T(x^k)-T(y^k)))=0$. For more details on strongly nonexpansive mappings see \cite{BruckReich1977jour} (Banach spaces) and \cite[Chapters 2--3]{Cegielski2012book} (Hilbert spaces).

Given a nonempty subset $A$ of $X$, its diameter is defined to be $\textnormal{diam}(A):=\sup\{\|a_1-a_2\|\,|\,a_1,a_2\in A\}\in [0,\infty]$. We say that $A$ is a half-space (respectively, a hyperplane) if there are a linear functional $f:X\to\R$ and $\alpha\in\R$ such that $A=\{x\in X\,| f(x)\leq \alpha\}$ (respectively, $A=\{x\in X\,| f(x)=\alpha\}$). We say that $A$ is an affine subspace if there is a point $p_A\in X$ and a linear subspace $\wt{A}$ of $X$ such that $A=p_A+\wt{A}$, and in this case $p_A$ is called the associated point and $\wt{A}$ is called the associated linear subspace. 

If $A$ is a closed affine subspace with an associated closed linear subspace $\wt{A}$ and an associated point $p_A$, then it can be assumed without loss of generality that $p_{A}$ is in the orthogonal complement $\wt{A}^{\perp}$ of $\wt{A}$, otherwise we use both the fact that $\wt{A}$ is closed under addition and the fact that $p_{A}=\wt{p}_{A}+p'_{A}$, where $\wt{p}_{A}\in \wt{A}$ and $p'_{A}\in\wt{A}^{\perp}$ to conclude from these facts that $A=p_{A}+\wt{A}=\wt{p}_{A}+p'_{A}+\wt{A}=p'_{A}+\wt{A}$, namely we have a representation of $A$ as a sum of $\wt{A}$ and a point from $\wt{A}^{\perp}$. Using this representation, the well-known identity $P_{p_A+\wt{A}}(\cdot)=p_A+P_{\wt{A}}(\cdot-p_{A})$ (see \cite[Proposition 3.19, p. 54]{BauschkeCombettes2017book}), the assumption $p_A\in \wt{A}^{\perp}$ and the linearity of $P_{\wt{A}}$, imply that $P_A=P_{\wt{A}}+(Id-P_{\wt{A}})(p_{A})=P_{\wt{A}}+p_A$. 

We denote by $d(x,A):=\inf\{\|x-a\|\,|\,a\in A\}$ the distance between the point $x\in X$ and the subset $A$, and by $dist(A_1,A_2):=\inf\{\|a_1-a_2\|\,|\,a_1\in A_1,\,a_2\in A_2\}$ the distance between the nonempty subsets $A_1\subseteq X$ and $A_2\subseteq X$. Given a natural number $m\in\N$ and a family  $\mathcal{L}:=(L_i)_{i=1}^m$ of $m$ closed linear subspaces $L_i$, $i\in I:=\{1,2,\ldots,m\}$ of $X$, we say that $\mathcal{L}$ is regular (or linearly regular) if there is $\eta>0$ such that $d(x,\cap_{i=1}^m L_i)\leq \eta \max\{d(x,L_i)\,|\, i\in I\}$ for all $x\in X$. We say that the family $\mathcal{L}$ is innately regular if every nonempty subfamily of $\mathcal{L}$ is regular. 

We say that a family $\mathcal{A}:=(A_i)_{i=1}^m$ of $m$ closed affine subspaces of $X$ is innately regular if the family $\wt{\mathcal{A}}:=(\wt{A}_i)_{i=1}^m$ of the associated closed linear subspaces $\wt{A}_i$, $i\in I$ is innately regular. See \cite{BauschkeTung2026prep,GunturkThao2023jour}
for more details related to innately regular families of closed linear subspaces, and \cite{Bauschke1995jour,BauschkeBorwein1996jour} for more details related to regular families of closed and convex sets (\cite{BauschkeBorwein1996jour} also presents various properties related to regular families of closed affine subspaces). In particular, as follows for example from the discussion presented in \cite[Section 2]{GunturkThao2023jour}, if each closed linear subspace $L_i$, $i\in I$ of the family $\mathcal{L}$ is either finite dimensional or finite co-dimensional, then $\mathcal{L}$ is innately regular. 
 
Finally, we note that, as is well known, many of the concepts mentioned above, such as the concept of a nonexpansive operator, carry over virtually word for word to any normed space and beyond, and in Theorem \bref{thm:UniformInfty} below we consider such a setting.

\section{Perturbation resilience}\label{sec:Perturbations}
This section presents two results related to Theorem \bref{thm:BauschkeTung} which take into account the possible existence of perturbations of the iterates of the iterative process. Allowing such perturbations may facilitate the study of the behavior of the iterative sequence when computational errors, noise, and so on arise. Moreover, the possibility to accommodate perturbations is valuable also for, and may be used in, the superiorization method \cite{CensorSuperiorizationPage,Censor2015surv,CensorDavidiHerman2010jour,CensorReem2015jour,
DavidiHermanCensor2009jour,Herman2014surv,ReemDe-Pierro2017jour}. In that method the property of perturbations resilience of a feasibility-seeking iterative process enables to interlace into the process objective function reduction steps without loosing the overall convergence to a feasible point of the imposed constraints.

\begin{thm}\label{thm:BoundedPerturbations}
In the setting of Theorem \bref{thm:BauschkeTung}, suppose that  $y^k:=x^k+e^k$ for all $k\in\N\cup\{0\}$, where  $(e^k)_{k=0}^{\infty}$ is any sequence of vectors in the given real Hilbert space $X$ whose magnitudes are bounded by some $\rho\in(0,\infty)$, i.e., $\|e^k\|\leq \rho$ for all $k\in\N\cup\{0\}$. Then $\|y^k\|\leq\|x^0\|+\lambda\mu_{\mathcal{A},\lambda}+\rho$ for each $k\in\N\cup\{0\}$. 
\end{thm}
\begin{proof}
Indeed, by Theorem \bref{thm:BauschkeTung} we know that $(x^k)_{k=0}^{\infty}$ is bounded by $\|x^0\|+\lambda\mu_{\mathcal{A},\lambda}$. Now, given an arbitrary metric space $(W,d)$, if $(u^k)_{k=0}^{\infty}$ is an arbitrary bounded sequence  there (located, say, in the closed ball of center $w\in W$ and radius $r>0$)  and if $(v^k)_{k=0}^{\infty}$ is any sequence which satisfies $\sigma:=\sup\{d(u^k,v^k)\,|\, k\in\N\cup\{0\}\}<\infty$, then also $(v^k)_{k=0}^{\infty}$ is bounded by $\sigma+r$ (since $d(v^k,w)\leq d(v^k,u^k)+d(u^k,w)\leq \sigma+r$ for all $k\in\N\cup\{0\}$). Thus, if for each $k\in\N\cup\{0\}$ we denote $u^k:=x^k$, $v^k:=y^k$, then $d(u^k,v^k)=\|e^k\|\leq \rho=:\sigma$. Since in our case $w:=0$ and  $r:=\|x^0\|+\lambda\mu_{\mathcal{A},\lambda}$, we conclude from the previous lines that $\|y^k\|\leq \sigma+r=\rho+\|x^0\|+\lambda\mu_{\mathcal{A},\lambda}$ for every $k\in\N\cup\{0\}$.   
\end{proof}

\begin{thm}\label{thm:SummablePerturbations}
In the setting of Theorem \bref{thm:BauschkeTung}, suppose that instead of \beqref{eq:x^(k+1)} we let
\begin{equation}\label{eq:x^(k+1)=x^(k)+e^k}
x^{k+1}:=T_k(x^k)+e^k, \quad \forall\, k\in\N\cup\{0\}, 
\end{equation}
where  $(e^k)_{k=0}^{\infty}$ is any sequence of vectors in $X$. Then $\|x^k\|\leq \|x^0\|+\lambda\mu_{\mathcal{A},\lambda}+\sum_{k=0}^{\infty}\|e^k\|$ for every $k\in\N\cup\{0\}$ (this upper bound is obviously of interest only when the perturbations are summable, namely when $\sum_{k=0}^{\infty}\|e^k\|<\infty$). 
\end{thm}
\begin{proof}
Since $A_i$ is a closed affine subspace for each $i\in I=\{1,2,\ldots,m\}$, for every $i\in I$ there exist a closed linear subspace $\wt{A}_i$ and a point $p_{A_i}\in \wt{A}_i^{\perp}$ such that $A_i=p_{A_i}+\wt{A}_i$, and we also have $P_{A_i}=P_{\wt{A_i}}+p_{A_i}$ (see Section \bref{sec:Preliminaries}). Therefore for all $k\in\N\cup\{0\}$ and all $x\in X$, we have 
\begin{multline}\label{eq:T_k_lambda_k}
T_k(x)=(1-\lambda_k)x+\lambda_k P_{A_{c(k)}}(x)=(1-\lambda_k)x+\lambda_k (P_{\wt{A}_{c(k)}}(x)+p_{A_{c(k)}})\\
=(1-\lambda_k)x+\lambda_k P_{\wt{A}_{c(k)}}(x)+\lambda_k p_{A_{c(k)}}=\wt{T}_k(x)+\lambda_k p_{A_{c(k)}},
\end{multline} 
where $\wt{T}_k(x):=(1-\lambda_k)x+\lambda_k P_{\wt{A}_{c(k)}}(x)$ for all $x\in X$ and all $k\in\N\cup\{0\}$. 
From \beqref{eq:x^(k+1)=x^(k)+e^k}, \beqref{eq:T_k_lambda_k}, the linearity of the operators $(\wt{T_k})_{k=0}^{\infty}$ and mathematical induction, we see that for each $k\in\N\cup\{0\}$, 
\begin{equation}\label{eq:x^(k+1)PertrubRecursion}
\begin{split}
x^{k+1}&=T_k(x^k)+e^k\\
&=\wt{T}_k(x^k)+\lambda_k p_{A_{c(k)}}+e^k\\
&=\wt{T}_k(T_{k-1}(x^{k-1})+e^{k-1})+\lambda_k p_{A_{c(k)}}+e^k\\
&=\wt{T}_k T_{k-1}(x^{k-1})+\wt{T}_k(e^{k-1})+\lambda_k p_{A_{c(k)}}+e^k\\
&=\wt{T}_k(\wt{T}_{k-1}(x^{k-1})+\lambda_{k-1}p_{A_{c(k-1)}})+\wt{T}_k(e^{k-1})+\lambda_k p_{A_{c(k)}}+e^k\\
&=\wt{T}_k\wt{T}_{k-1}(x^{k-1})+\lambda_{k-1}\wt{T}_k(p_{A_{c(k-1)}})+\wt{T}_k(e^{k-1})+\lambda_k p_{A_{c(k)}}+e^k\\
&=
\wt{T}_k\wt{T}_{k-1}(T_{k-2}(x^{k-2})+e^{k-2})+\lambda_{k-1}\wt{T}_k(p_{A_{c(k-1)}})+\wt{T}_k(e^{k-1})+\lambda_k p_{A_{c(k)}}+e^k\\
&=\wt{T}_k\wt{T}_{k-1}(T_{k-2}(x^{k-2}))+\wt{T}_k\wt{T}_{k-1}(e^{k-2})+\lambda_{k-1}\wt{T}_k(p_{A_{c(k-1)}})+\wt{T}_k(e^{k-1})+\lambda_k p_{A_{c(k)}}+e^k\\
&=\wt{T}_k\wt{T}_{k-1}(\wt{T}_{k-2}(x^{k-2})+\lambda_{k-2}p_{A_{c(k-2)}})+\wt{T}_k\wt{T}_{k-1}(e^{k-2})+\lambda_{k-1}\wt{T}_k(p_{A_{c(k-1)}})+\wt{T}_k(e^{k-1})\\
&+\lambda_k p_{A_{c(k)}}+e^k\\
&=\wt{T}_k\wt{T}_{k-1}\wt{T}_{k-2}(x^{k-2})+\lambda_{k-2}\wt{T}_k\wt{T}_{k-1}(p_{A_{c(k-2)}})+\wt{T}_k\wt{T}_{k-1}(e^{k-2})+\lambda_{k-1}\wt{T}_k(p_{A_{c(k-1)}})\\
&+\wt{T}_k(e^{k-1})+\lambda_k p_{A_{c(k)}}+e^k\\
&=\ldots\\
&=\wt{T}_{k}\wt{T}_{k-1}\cdots \wt{T}_0(x^0)+\sum_{j=0}^k \wt{T}_k\cdots \wt{T}_{k+1-j}(e^{k-j})+\sum_{j=0}^k \lambda_{k-j}\wt{T}_k\cdots \wt{T}_{k+1-j}(p_{A_{c(k-j)}}),
\end{split}
\end{equation}
where in the last line of \beqref{eq:x^(k+1)PertrubRecursion} we use the convention $\wt{T}_k\cdots \wt{T}_{k+1-j}(e^{k-j}):=e^k$ and $\wt{T}_k\cdots \wt{T}_{k+1-j}(p_{A_{c(k-j)}}):=p_{A_{c(k)}}$ for $j:=0$. 

We observe that in the particular case where $e^k=0$ for all $k\in\N\cup\{0\}$, then \beqref{eq:x^(k+1)PertrubRecursion} and the linearity of the operators $\wt{T}_k$, $k\in \N\cup\{0\}$ imply the equality $x^{k+1}=\wt{T}_{k}\wt{T}_{k-1}\cdots \wt{T}_0(x^0)+\sum_{j=0}^k \lambda_{k-j}\wt{T}_k\cdots \wt{T}_{k+1-j}(p_{A_{c(k-j)}})$. On the other hand, in this case we are fully in the setting of Theorem \bref{thm:BauschkeTung}, and so this theorem implies that 
\begin{equation}\label{eq:sum_lambda_(k-j)T_kT_(k+1-j)p_A_c(k-j)}
\|\wt{T}_{k}\wt{T}_{k-1}\cdots \wt{T}_0(x^0)+\sum_{j=0}^k \lambda_{k-j}\wt{T}_k\cdots \wt{T}_{k+1-j}(p_{A_{c(k-j)}})\|\leq \|x^0\|+\lambda\mu_{\mathcal{A},\lambda}.
\end{equation}
In addition, the triangle inequality, together with the fact that $T_j$ is linear and nonexpansive for each $j\in \{0,1,\ldots,k\}$, imply that 
\begin{multline}\label{eq:sum e^j}
\|\sum_{j=0}^k \wt{T}_k\cdots \wt{T}_{k+1-j}(e^{k-j})\|\leq \|e^k\|+\sum_{j=1}^k\|\wt{T}_k\cdots \wt{T}_{k+1-j}(e^{k-j})\|\\
\leq \|e^k\|+\sum_{j=1}^k\|e^{k-j}\|\leq \sum_{j=0}^{\infty}\|e^j\|.
\end{multline}
The combination of \beqref{eq:x^(k+1)PertrubRecursion}, \beqref{eq:sum_lambda_(k-j)T_kT_(k+1-j)p_A_c(k-j)}, \beqref{eq:sum e^j} and the triangle inequality imply the assertion (even if $\sum_{j=0}^{\infty}\|e^j\|=\infty$, but in this case the estimate is trivial).
\end{proof}

\section{Closed half-spaces}\label{sec:Halfspaces}
In this section we establish Theorem \bref{thm:Halfspaces}, which shows that Theorem \bref{thm:BauschkeTung}
can be extended to closed half-spaces. This result partly extends \cite[Theorem 3.2]{BauschkeTung2025prep} in the sense that it allows a mixture of closed half-spaces and hyperplanes in $\mathcal{A}$. Note that a relaxed projection of a point $x\in X$ onto a closed half-space $A$ does not necessarily coincide with the relaxed projection onto the boundary hyperplane $\wh{A}$ of $A$ (they coincide if $x\notin A$ or if $x\in\wh{A}$); hence Theorem \bref{thm:Halfspaces} below does not follow immediately from Theorem \bref{thm:BauschkeTung}.  

\begin{thm}\label{thm:Halfspaces}
Consider the setting of Theorem \bref{thm:BauschkeTung}, wherein the family $\mathcal{A}:=(A_i)_{i=1}^m$ has the property that each $A_i$, $i\in\{1,2,\ldots,m\}$ is either a closed half-space or a closed hyperplane. Generate the sequence $(x^k)_{k=0}^{\infty}$ with the relaxation sequence of operators $(T_k)_{k=0}^{\infty}$ as in \beqref{eq:x^(k+1)}. Let  $\wh{A}_i$  be the boundary hyperplane of $A_i$ if $A_i$ is a half-space, and $\wh{A}_i:=A_i$ if $A_i$ is a hyperplane, and let $\wh{\mathcal{A}}:=(\wh{A}_i)_{i=1}^m$ (this is an innately regular family since each $\wh{A}_i$ is finite co-dimensional). Then $\|x^k\|\leq \|x^0\|+\lambda \mu_{\wh{\mathcal{A}},\lambda}$ for all $k\in\N\cup\{0\}$, namely  $(x^k)_{k=0}^{\infty}$ is bounded by the bound mentioned in Theorem \bref{thm:BauschkeTung}, where instead of $\mathcal{A}$ there one takes $\wh{\mathcal{A}}$.    
\end{thm}
\begin{proof}
There are two possibilities: in the first one the sequence $(x^k)_{k=0}^{\infty}$ becomes constant starting from some place onward, and in the second one $(x^k)_{k=0}^{\infty}$ never becomes constant starting from any place. In what follows we consider each of these possibilities separately, even though the proofs in both cases share many similarities. \\\vspace*{0.2cm}

{\noindent\bf The first possibility:} In this case there is a minimal $\wt{k}\in\N\cup\{0\}$ such that $x^k=x^{\wt{k}}$ for all $\wt{k}\leq k\in\N\cup\{0\}$. Therefore, $(x^k)_{k=0}^{\infty}$ is bounded by $\max_{k\in\{0,\ldots,\wt{k}\}}\|x^k\|$. The rest of the proof in the first possibility is aimed at showing that the bound $\max_{k\in\{0,\ldots,\wt{k}\}}\|x^k\|$ is by itself bounded by the bound mentioned in Theorem \bref{thm:BauschkeTung}, where instead of $\mathcal{A}$ there one takes $\wh{\mathcal{A}}$. 

The minimality of $\wt{k}$ implies that the finite sequence $(x^k)_{k=0}^{\wt{k}}$ does not become constant starting from some place $u\in \{0,1,\ldots,\wt{k}-1\}$  (where $\{0,1,\ldots,\wt{k}-1\}:=\emptyset$ if $\wt{k}=0$), and hence we can construct by induction a finite subsequence $(x^{k_j})_{j=0}^{q}$, such that $q\in \N\cup\{0\}$  is not bigger than $\wt{k}$, $k_0=0$, $k_q=\wt{k}$,  $k_j<k_{j+1}$ and $x^{k_{j+1}}\neq x^{k_j}$ for all $j\in\{0,1,\ldots,q-1\}$ (here and elsewhere $\{0,1,\ldots,q-1\}:=\emptyset$ if $q=0$), and for all $j\in\{0,\ldots,q-1\}$ and all $s\in\{k_j,k_j+1,\ldots,k_{j+1}-1\}$ we have  $x^s=x^{k_j}$. 

Since $x^s=x^{k_j}$ for all $j\in\{0,1,\ldots,q-1\}$ and all $s\in\{k_j,k_j+1,\ldots,k_{j+1}-1\}$, and since $x^{k+1}=T_{k}(x^k)$ for each $k\in\N\cup\{0\}$, we have 
\begin{equation}
x^{k_{j+1}}=T_{k_{j+1}-1}(x^{k_{j+1}-1})=\ldots=T_{k_{j+1}-1}(x^{k_{j}}).
\end{equation}
If $x^{k_j}\in A_{c(k_{j+1}-1)}$ for some $j\in\{0,1,\ldots,q-1\}$, then $x^{k_j}=P_{A_{c(k_{j+1}-1)}}(x^{k_j})$, and since $x^{k_{j+1}}=T_{k_{j+1}-1}(x^{k_{j}})$ as was shown above, we have 
\begin{equation*}
x^{k_{j+1}}=T_{k_{j+1}-1}(x^{k_{j}})=(1-\lambda_{k_{j+1}-1})x^{k_j}+\lambda_{k_{j+1}-1} P_{A_{c(k_{j+1}-1)}}(x^{k_j})=x^{k_j},
\end{equation*}
a contradiction to the construction of $(x^{k_t})_{t=0}^{q}$. Therefore, $x^{k_j}\notin A_{c(k_{j+1}-1)}$ for all $j\in\{0,\ldots,q-1\}$. Now, either $A_{c(k_{j+1}-1)}$ is a half-space, and then the relation $x^{k_j}\notin A_{c(k_{j+1}-1)}$ implies that $T_{k_{j+1}-1}(x^{k_{j}})$  is the relaxed projection of $x^{k_{j}}$ onto the boundary hyperplane $\wh{A}_{c(k_{j+1}-1)}$ of $A_{c(k_{j+1}-1)}$, or $A_{c(k_{j+1}-1)}$ is a hyperplane, and then $T_{k_{j+1}-1}(x^{k_{j}})$ is the relaxed projection of $x^{k_{j}}$ onto this hyperplane, which is equal to $\wh{A}_{c(k_{j+1}-1)}$ by the definition of $\wh{A}_{c(k_{j+1}-1)}$.

Define a control $\wh{c}:\N\cup\{0\}\to I$ by $\wh{c}(j):=c(k_{j+1}-1)$ for all $j\in\{0,\ldots,q-1\}$, and $\wh{c}(j)=c(j)$ for every $j\geq q$. In addition, let $\wh{T}_j:=T_{k_{j+1}-1}$ and $\wh{x}^j:=x^{k_j}$ for all $j\in\{0,\ldots,q-1\}$ (and $\wh{T}_0:=T_0$ and $\wh{x}^0:=x^0$ if $q=0$), and $\wh{T}_j:=T_j$ and $\wh{x}^j:=\wh{T}_{j-1}(\wh{x}^{j-1})$ for all $q\leq j\in\N$. 

From previous lines above we conclude that for all $j\in\{0,1,\ldots,q-1\}$
\begin{equation}
\wh{x}^{j+1}=x^{k_{j+1}}=T_{k_{j+1}-1}(x^{k_{j}})=\wh{T}_{j}(\wh{x}^j).
\end{equation} 
Since $\wh{x}^{j+1}=\wh{T}_j(\wh{x}^j)$ also for all $q\leq j\in \N\cup\{0\}$ by the definitions of $\wh{x}^j$ and $\wh{T}_j$, we conclude from Theorem \bref{thm:BauschkeTung} (in which we take $\wh{\mathcal{A}}$ instead of $\mathcal{A}$, $\wh{c}$ instead of $c$, $(\wh{T}_k)_{k=0}^{\infty}$ instead of $(T_k)_{k=0}^{\infty}$, and $(\wh{x}^k)_{k=0}^{\infty}$ instead of $(x^k)_{k=0}^{\infty}$) that $(\wh{x}^k)_{k=0}^{\infty}$ is bounded by $\|\wh{x}^0\|+\lambda \mu_{\wh{\mathcal{A}},\lambda}$. Since the set $\{\|x^k\|\,|\, k\in\{0,1,\ldots,\wt{k}\}\}$ is equal to the set $\{\|\wh{x}^j\|\,|\, j\in\{0,1,\ldots,q\}\}$ by the construction of $(\wh{x}^j)_{j=0}^{q}$, and since 
\begin{equation}
\sup\{\|\wh{x}^j\|\,|\, j\in\{0,1,\ldots,q\}\}\leq \sup\{\|\wh{x}^j\|\,|\, j\in\N\cup\{0\}\leq \|\wh{x}^0\|+\lambda \mu_{\wh{\mathcal{A}},\lambda}, 
\end{equation}
it follows that the finite subsequence $(x^k)_{k=0}^{\wt{k}}$ is  bounded by the same bound $\|\wh{x}^0\|+\lambda \mu_{\wh{\mathcal{A}},\lambda}$, which is equal to $\|x^0\|+\lambda \mu_{\wh{\mathcal{A}},\lambda}$ since $\wh{x}^0=x^0$. \\\vspace*{0.2cm}

{\noindent\bf The second possibility:} Since in this case $(x^k)_{k=0}^{\infty}$ never becomes constant from some place onward, we can construct, by induction, an infinite subsequence $(x^{k_j})_{j=0}^{\infty}$ having the properties that $k_0=0$, $k_j<k_{j+1}$ and $x^{k_{j+1}}\neq x^{k_j}$ for all $j\in\N\cup\{0\}$, and also that for all $j\in\N\cup \{0\}$ and all $s\in\{k_j,k_j+1,\ldots,k_{j+1}-1\}$ we have  $x^s=x^{k_j}$. 
 
Since $x^s=x^{k_j}$ for all $j\in\N\cup\{0\}$ and all $s\in\{k_j,k_j+1,\ldots,k_{j+1}-1\}$, and since $x^{k+1}=T_{k}(x^k)$ for each $k\in\N\cup\{0\}$, we have 
\begin{equation}
x^{k_{j+1}}=T_{k_{j+1}-1}(x^{k_{j+1}-1})=\ldots=T_{k_{j+1}-1}(x^{k_{j}}).
\end{equation}
If $x^{k_j}\in A_{c(k_{j+1}-1)}$ for some $j\in\N\cup\{0\}$, then $x^{k_j}=P_{A_{c(k_{j+1}-1)}}(x^{k_j})$, and since $x^{k_{j+1}}=T_{k_{j+1}-1}(x^{k_{j}})$ as follows from previous lines, we have 
\begin{equation*}
x^{k_{j+1}}=T_{k_{j+1}-1}(x^{k_{j}})=(1-\lambda_{k_{j+1}-1})x^{k_j}+\lambda_{k_{j+1}-1} P_{A_{c(k_{j+1}-1)}}(x^{k_j})=x^{k_j},
\end{equation*}
a contradiction to the construction of $(x^{k_t})_{t=0}^{\infty}$. Therefore, $x^{k_j}\notin A_{c(k_{j+1}-1)}$ for all $j\in\N\cup\{0\}$. Now, either $A_{c(k_{j+1}-1)}$ is a half-space, and then the relation $x^{k_j}\notin A_{c(k_{j+1}-1)}$ implies that $T_{k_{j+1}-1}(x^{k_{j}})$  is the relaxed projection of $x^{k_{j}}$ onto the boundary hyperplane $\wh{A}_{c(k_{j+1}-1)}$ of $A_{c(k_{j+1}-1)}$, or $A_{c(k_{j+1}-1)}$ is a hyperplane, and then $T_{k_{j+1}-1}(x^{k_{j}})$ is the relaxed projection of $x^{k_{j}}$ onto this hyperplane, which is equal to $\wh{A}_{c(k_{j+1}-1)}$ by the definition of $\wh{A}_{c(k_{j+1}-1)}$.

Define a control $\wh{c}:\N\cup\{0\}\to I$ by $\wh{c}(j):=c(k_{j+1}-1)$ for all $j\in\N\cup\{0\}$. In addition, let $\wh{T}_j:=T_{k_{j+1}-1}$ and $\wh{x}^j:=x^{k_j}$ for all $j\in\N\cup\{0\}$. From previous lines we conclude that 
\begin{equation}
\wh{x}^{j+1}=x^{k_{j+1}}=T_{k_{j+1}-1}(x^{k_{j}})=\wh{T}_{j}(\wh{x}^j)\end{equation}
for all $j\in\N\cup\{0\}$. Thus, we conclude from Theorem \bref{thm:BauschkeTung} (in which we take $\wh{\mathcal{A}}$ instead of $\mathcal{A}$, $\wh{c}$ instead of $c$, $(\wh{T}_k)_{k=0}^{\infty}$ instead of $(T_k)_{k=0}^{\infty}$, and $(\wh{x}^k)_{k=0}^{\infty}$ instead of $(x^k)_{k=0}^{\infty}$) that $(\wh{x}^k)_{k=0}^{\infty}$ is bounded by $\|\wh{x}^0\|+\lambda \mu_{\wh{\mathcal{A}},\lambda}$. Since the set $\{\|x^k\|\,|\, k\in\N\cup\{0\}\}$ is equal to the set $\{\|\wh{x}^j\|\,|\, j\in\N\cup\{0\}\}$ by the construction of $(\wh{x}^j)_{j=0}^{\infty}$, it follows that the sequence $(x^k)_{k=0}^{\infty}$ is bounded by $\|\wh{x}^0\|+\lambda \mu_{\wh{\mathcal{A}},\lambda}$, which is equal to $\|x^0\|+\lambda \mu_{\wh{\mathcal{A}},\lambda}$ since $\wh{x}^0=x^0$.
\end{proof}

\section{Dynamic string-averaging of relaxed projections}\label{sec:Dynamic-string-averaging}
In this section  we consider the dynamic string-averaging algorithmic scheme, introduced in \cite{CensorElfvingHerman2001incol}, and show that Theorem \bref{thm:BauschkeTung} can be generalized to sequences generated by dynamic weighted sums of dynamic strings (of dynamic lengths) of relaxed projections. Dynamic string-averaging (of a more restricted form) is used in various algorithmic schemes, such as  \cite{BargetzReichZalas2018jour,BrookeCensorGibali2023jour,CensorNisenbaum2021jour,CensorSegal2009jour,CensorTom2003jour,
CensorZaslavski2013jour,PenfoldSchulteCensorBashkirovMcAllisterSchubertRosenfeld2010incol}.

\begin{thm}\label{thm:BoundedString}
Suppose that $X$ is a real Hilbert space, $m\in\N$ and $\lambda\in (0,2)$. Suppose further that $\mathcal{A}:=(A_i)_{i=1}^m$ is a family of $m$ closed affine subspaces in $X$ which is innately regular. Let $\mu_{\mathcal{A},\lambda}$ be the positive constant associated with $\lambda$ and $\mathcal{A}$, as ensured by Theorem \bref{thm:BauschkeTung}. For each $k\in\N\cup\{0\}$ assume that $J_k$ is a nonempty finite set and $w_k:J_k\to [0,1]$ is a weight function. For each $k\in\N\cup\{0\}$ and  each $j\in J_k$ let $\ell_k(j)\in\N$ be an arbitrary natural number, and consider an arbitrary control $t_k:\{(j,1),(j,2),\ldots,(j,\ell_k(j))\}\to I:=\{1,2,\ldots,m\}$. For each $k\in\N\cup\{0\}$, $j\in J_k$, and $h\in\{1,2,\ldots,\ell_k(j)\}$ let $\lambda_{k,h}\in [0,\lambda]$ be arbitrary and let the operator $T_{k,j,h}:X\to X$ be defined as
\begin{equation}
T_{k,j,h}:=(1-\lambda_{k,h})\Id+\lambda_{k,h}P_{A_{t_k(j,h)}},
\end{equation}
namely $T_{k,j,h}(x)$ is the relaxed projection of $x\in X$ onto the closed affine subspace  $A_{t_k(j,h)}$ with relaxation coefficient $\lambda_{k,h}$. For each $k\in\N\cup\{0\}$ and $j\in J_k$, denote the product 
\begin{equation}\label{eq:S(k,j)}
S(k,j):=T_{k,j,\ell_k(j)}T_{k,j,\ell_k(j)-1}\cdots T_{k,j,1},
\end{equation}
namely $S(k,j)$ is the relaxed string operator obtained by composing the operators $T_{k,j,h}$ in the above order (from $h:=1$ to $h:=\ell_{k}(j)$, where $T_{k,j,1}$ is the right-most factor and $T_{k,j,\ell_k(j)}$ is the left-most factor). For an arbitrary $x^0\in X$ define recursively the sequence $(x^k)_{k=0}^{\infty}$ by the iterative process 
\begin{equation}\label{eq:string}
x^{k+1}:=\sum_{j\in J_k}w_k(j)S(k,j)(x^k), \quad \forall k\in\N\cup\{0\}.
\end{equation}
In other words, $x^{k+1}$ is a weighted sum of string operators, where the weights, the strings and the lengths of the strings vary with $k\in\N\cup\{0\}$ and $j\in J_k$. 

Then $\|x^k\|\leq \|x^0\|+\lambda \mu_{\mathcal{A},\lambda}$ for each $k\in\N\cup\{0\}$, namely $(x^k)_{k=0}^{\infty}$ is bounded by the same bound which appears in Theorem \bref{thm:BauschkeTung} in relation to the relaxation coefficient $\lambda$ and the innately regular family $\mathcal{A}$.
\end{thm}

\begin{remark}
In the setting of Theorem \bref{thm:BoundedString}, if we start with some arbitrary control $c:\N\cup\{0\}\to I$, let $J_k:=\{1\}$, $\ell_k(1):=1$ and $t_k(1,1):=c(k)$  for each $k\in\N\cup\{0\}$, then $w_k(j)=1$ and $S(k,j)=T_{k,1,1}=(1-\lambda_{k,1})\Id+\lambda_{k,1}P_{A_{c(k)}}$ for each $k\in\N\cup\{0\}$ and each $j\in J_k$; thus, by denoting $T_k:=T_{k,1,1}$ we have $x^{k+1}=T_{k,1,1}x^k$ for each $k\in\N\cup\{0\}$, namely we obtain the setting of Theorem \bref{thm:BauschkeTung} which is a sequential iterative scheme. 

As another example, if we let $J_k:=\{1,2,\ldots,m\}$, $\ell_k(1):=1$ and $t_k(j,1):=j$ for each $k\in\N\cup\{0\}$ and each $j\in J_k$, and if for all $k\in\N\cup\{0\}$ we take $w_k:J_k\to [0,1]$ to be any weight function, then $T_{k,j,\ell_k(j)}=T_{k,j,1}=(1-\lambda_{k,1})\Id+\lambda_{k,1}P_{A_j}$ and $x^{k+1}=\sum_{j=1}^m w_k(j)T_{k,j,1}(x^k)$, namely this is a so-called simultaneous iterative scheme. 

As a third example, let $J_k:=\{0,\ldots,k\}$ and $\ell_k(j):=j+1$ for each $k\in\N\cup\{0\}$ and each $j\in J_k$. Moreover, let $t_k:\{(j,1),(j,2),\ldots,(j,j+1)\}\to I$ be arbitrary and let $w_k(j):=2(j+1)/((k+1)(k+2))>0$ for every $k\in\N\cup\{0\}$ and $j\in J_k$. Then indeed $\sum_{j\in J_k}w_k(j)=1$ and 
\begin{equation*}
x^{k+1}=\sum_{j=0}^k \frac{2(j+1)}{(k+1)(k+2)}T_{k,j,j+1} T_{k,j,j}\ldots T_{k,j,1}(x^k)
\end{equation*}
for each $k\in\N\cup\{0\}$. Many other examples can be given, and this shows the generality of Theorem \bref{thm:BoundedString}.
\end{remark}

In order to prove Theorem \bref{thm:BoundedString}, we need the following simple and known lemma.

\begin{lem}\label{lem:LinearConvex}
Given a real or complex vector space $X$, the set of all linear convex operators (see Section \bref{sec:Preliminaries} for the definition of this concept) from $X$ to itself is closed under addition, multiplication by scalar, and composition. This set contains the set of affine operators. Moreover, for all linear convex operator $T:X\to X$ and any finite number $q\in\N$ of numbers $\alpha_1,\ldots,\alpha_{q}\in [0,1]$  and vectors $x_1,\ldots,x_{q}\in X$, one has $T(\sum_{j=1}^q\alpha_jx_j)=\sum_{j=1}^q\alpha_jT(x_j)$.
\end{lem}
\begin{proof}
The proof is immediate from the definition and mathematical induction.
\end{proof}

\begin{proof}[Proof of Theorem \bref{thm:BoundedString}]
Lemma \bref{lem:LinearConvex} and \beqref{eq:S(k,j)} ensure that $S(k,j)$ is linear convex. Since from \beqref{eq:string} we have $x^1=\sum_{j_0\in J_0}w_0(j_0)S(0,j_0)(x^0)$, and since $w_0$ is a weight function, it follows from Lemma \bref{lem:LinearConvex} again that 
\begin{align*}
x^2 &=\sum_{j_1\in J_1}w_1(j_1)S(1,j_1)(x^1)\\
&=\sum_{j_1\in J_1}w_1(j_1)S(1,j_1)\left(\sum_{j_0\in J_0}w_0(j_0)S(0,j_0)(x^0)\right)\\
&=\sum_{j_1\in J_1}\sum_{j_0\in J_0}w_1(j_1)w_0(j_0)S(1,j_1)(S(0,j_0)(x^0)).
\end{align*}
Similarly, Lemma \bref{lem:LinearConvex} and induction imply that for all $k\in\N\cup\{0\}$,  
\begin{align}\label{eq:sumS(j_k,j_0)}
x^{k+1}&=\sum_{j_k\in J_k}\ldots\sum_{j_0\in J_0}w_k(j_k)\ldots w_0(j_0)S(j_k,\ldots,j_0)(x^0),
\end{align}
where $S(j_k,\ldots,j_0)$ is the string operator defined as 
\begin{align*}
S(j_k,\ldots,j_0)&:=S(k,j_k)S(k-1,j_{k-1})\cdots S(0,j_0)\\
&=(T_{k,j_k,\ell_k(j_k))}\cdots T_{k,j_k,1}) \cdots (T_{k,j_0,\ell_0(j_0)} \cdots T_{k,j_0,1})
\end{align*}
for all $k\in\N\cup\{0\}$ and all $(j_k,\ldots,j_0)\in J_k\times\ldots \times J_0$. Since $S(j_k,\ldots,j_0)$ is a finite product of operators (of length $\ell_k(j_k)+\ell_{k-1}(j_{k-1})+\ldots+\ell_{0}(j_0)$) from the set $\mathcal{R}_{\mathcal{A},\lambda}$ (see \beqref{eq:R_A_lambda}), it follows from Theorem \bref{thm:BauschkeTung} that $\|S(j_k,\ldots,j_0)(x^0)\|\leq \|x^0\|+\lambda \mu_{\mathcal{A},\lambda}$ for all $k\in\N\cup\{0\}$ and all $(j_k,\ldots,j_0)\in J_k\times\ldots \times J_0$. This inequality, \beqref{eq:sumS(j_k,j_0)}, the triangle inequality, induction, and the fact that $w_k$ is a weight function for each $k\in\N\cup\{0\}$, all imply that 
\begin{align*}
\|x^{k+1}\|&\leq \sum_{j_k\in J_k}\ldots\sum_{j_0\in J_0}w_k(j_k)\cdots w_0(j_0)\|S(j_k,\ldots,j_0)(x^0)\|\\
&\leq  \sum_{j_k\in J_k}\ldots\sum_{j_0\in J_0}w_k(j_k)\cdots w_0(j_0)(\|x^0\|+\lambda \mu_{\mathcal{A},\lambda})\\
&=\sum_{j_k\in J_k}\ldots\sum_{j_1\in J_1}w_k(j_k)\cdots w_1(j_1)\left((\|x^0\|+\lambda \mu_{\mathcal{A},\lambda})\sum_{j_0\in J_0}w_0(j_0)\right)\\
&= \sum_{j_k\in J_k}\ldots\sum_{j_1\in J_1}w_k(j_k)\cdots w_1(j_1)((\|x^0\|+\lambda \mu_{\mathcal{A},\lambda})\cdot 1)\\
&=\ldots=(\|x^0\|+\lambda \mu_{\mathcal{A},\lambda})\sum_{j_k\in J_k}w_k(j_k)=\|x^0\|+\lambda \mu_{\mathcal{A},\lambda}
\end{align*}
 for every $k\in\N\cup\{0\}$, as claimed.
\end{proof}

\section{Limitations on the possibility to extend Theorem \bref{thm:BauschkeTung}}\label{sec:Limitations-on-the}
This section shows, by means of several examples, that Theorem \bref{thm:BauschkeTung} cannot be extended without further assumptions to arbitrary nonempty, closed and convex subsets. A different example, for the case where $\mathcal{A}$ was a family of two closed affine subspaces and the operators were orthogonal projections, was given in \cite[Example 5.3]{BauschkeTung2026prep}. We explain below (Example \bref{ex:Aleph}) why there are at least $2^{{\aleph}_0}$ examples in which the associated control is non-cyclic and not almost cyclic. We need a few auxiliary results before presenting the examples.

\begin{lem}\label{lem:BruckreichCorollary1.4}
If $X$ is a real Hilbert space and $T:X\to X$ is strongly nonexpansive, then $T$ is fixed point free if and only if $\lim_{k\to\infty}\|T^k(x)\|=\infty$ for all $x\in X$. 
\end{lem}
\begin{proof}
This known lemma is just a special case of a result of Bruck and Reich \cite[Corollary 1.4]{BruckReich1977jour}, where, in the notation of that result,  $C:=X$. 

Indeed, the intersection of $C$ with any closed ball is the  closed ball itself which is a weakly compact and convex subset \cite[Theorem 3.37]{BauschkeCombettes2017book}, and so $C$ is a nonempty boundedly weakly compact and convex subset of $X$. Now, since in a Hilbert space a convex subset is weakly compact if and only if it is closed and bounded (as follows from \cite[Lemma 2.36, Theorem 3.34, Theorem 3.37]{BauschkeCombettes2017book}), the well-known Browder-G{\"o}hde-Kirk theorem \cite[Theorem 4.29]{BauschkeCombettes2017book} ensures that any nonexpansive mapping from a nonempty, convex and bounded subset to itself has a fixed point, namely any nonempty, convex and bounded subset of $C$ has the fixed point property for nonexpansive mappings.

Thus, all the conditions stated in \cite[Corollary 1.4]{BruckReich1977jour} are fulfilled, and one concludes that if $T$ is fixed point free, then $\lim_{k\to\infty}\|T^k(x)\|=\infty$ for all $x\in X$. The converse direction obviously holds since if $\lim_{k\to\infty}\|T^k(x)\|=\infty$ for all $x\in X$, then the fixed point set of $T$ must be empty, otherwise, if $x$ is a fixed of $T$, then $x=T(x)=\ldots=T^k(x)$ for all $k\in\N$ and so $\lim_{k\to\infty}\|T^k(x)\|=\|x\|<\infty$, a contradiction. 
\end{proof}

\begin{thm}\label{thm:UniformInfty}
Suppose that $(X,\|\cdot\|)$ is a normed space and let $I$ be a nonempty set of indices (finite or infinite). Let $(T_i)_{i\in I}$ be a family of nonexpansive mappings from $X$ to itself. These mappings may coincide. Let $c:\N\cup\{0\}\to I$ be an arbitrary control. Consider the product sequence $(S_k)_{k=0}^{\infty}$ of operators defined by $S_0:=T_{c(0)}$, $S_1:=T_{c(1)}T_{c(0)}$, $S_2:=T_{c(2)}T_{c(1)}T_{c(0)}$, and so on, that is, for each $k\in\N\cup\{0\}$, we have $S_{k}:=T_{c(k)}T_{c(k-1)}\cdots T_{c(0)}$. 
\begin{enumerate}[(a)]
\item\label{item:BoundedSome_x} If $(S_k(x))_{k=0}^{\infty}$ is bounded for some $x\in X$, then $(S_k)_{k=0}^{\infty}$ is uniformly bounded on all nonempty and bounded subsets of $X$, namely for each nonempty and bounded set  $B\subseteq X$ there exists some $r_B>0$ such that $\|S_k(y)\|<r_B$ for all $k\in\N\cup\{0\}$ and all $y\in B$. In particular, $(S_k(y))_{k=0}^{\infty}$ is bounded for all $y\in X$.

\item $(S_k(x))_{k=0}^{\infty}$ is bounded for some $x\in X$ if and only if $(S_k(x))_{k=0}^{\infty}$ is bounded for all $x\in X$.
 
\item\label{item:UnBoundedSome_x} If $\lim_{k\to\infty}\|S_{k}(x)\|=\infty$ for some $x\in X$, then  actually $\lim_{k\to\infty}\|S_{k}(\cdot)\|=\infty$ uniformly on nonempty and bounded sets, that is, for each nonempty and bounded set $B\subseteq X$ and for each $r>0$, there exists some $\wt{k}\in\N$ such that $\|S_{k}(y)\|>r$ for all $\wt{k}<k\in\N\cup\{0\}$ and all $y\in B$. In particular, $\lim_{k\to\infty}\|S_{k}(y)\|=\infty$ for all $y\in X$. 

\item $(S_k(x))_{k=0}^{\infty}$ is unbounded for some $x\in X$ if and only if $(S_k(x))_{k=0}^{\infty}$ is unbounded for all $x\in X$.

\item\label{item:FixedPointFree} If $X$ is a real Hilbert space and $T: X\to X$ is strongly nonexpansive and is fixed point free, then $\lim_{k\to \infty}\|T^k(\cdot)\|=\infty$ uniformly on nonempty and bounded subsets of $X$. 
\end{enumerate} 
\end{thm}

\begin{proof}
\begin{enumerate}[(a)]
\item We first prove the assertion for the particular case of bounded subsets of $X$ which contain $x$. Suppose, by way of contradiction, that this particular case of the assertion is false. Then there exists some bounded set $B$  which contains $x$ such that for all $r>0$, in particular for all $r\in\N$, there exists some $k_r\in\N$ and $y_{k_r}\in B$ such that $\|S_{k_r}(y_{k_r})\|\geq r+\textnormal{diam}(B)$ (since $B$ is bounded, its diameter $\textnormal{diam}(B)$ is finite).

Because $T_j$ is nonexpansive for each $j\in I$, also any finite composition of these mappings is a nonexpansive mapping, and so $S_k$ is nonexpansive for each $k\in\N\cup\{0\}$. Since both $x$ and $y_{k_r}$ are in $B$, we have $\|S_{k_r}(y_{k_r})-S_{k_r}(x)\|\leq \|y_{k_r}-x\|\leq \textnormal{diam}(B)<\infty$. The previous inequalities and the triangle inequality imply that
\begin{align*}
\begin{split}
r+\textnormal{diam}(B)\leq \|S_{k_r}(y_{k_r})\|&\leq \|S_{k_r}(x)\|+\|S_{k_r}(y_{k_r})-S_{k_r}(x)\|\\
&\leq\|S_{k_r}(x)\|+\textnormal{diam}(B).
\end{split}
\end{align*}
As a result, for each $r\in\N\cup\{0\}$ we can find some $k_r\in\N\cup\{0\}$ such that $r\leq \|S_{k_r}(x)\|$, and so the sequence $(S_k(x))_{k=0}^{\infty}$ is not bounded, a contradiction to what we assumed in the formulation of the assertion. Thus, the particular case of the assertion concerning bounded subsets of $X$ which contain $x$ is true. 

If $B$ is nonempty and bounded and does not contain $x$, then $B\cup\{x\}$ is non-empty and bounded and contains $x$, and so the assertion holds for it by what we showed above. Hence there is $\wt{r}>0$ such that $\|S_k(y)\|<\wt{r}$ for all $y\in B\cup\{x\}$ and in particular for all $y\in B$, as required. Finally, given $y\in X$, if we denote $B:=\{y\}$, then $B$ is bounded and what we showed above shows that $(S_k(y))_{k=0}^{\infty}$ is bounded. 

\item Obviously if $(S_k(x))_{k=0}^{\infty}$ is bounded for all $x\in X$, then it is bounded for some $x\in X$. On the other hand, if  $(S_k(x))_{k=0}^{\infty}$ is bounded for some $x\in X$, then, from Part \beqref{item:BoundedSome_x}, it is bounded for all $x\in X$.

\item We first prove the assertion for bounded subsets $B$ of $X$ which contain $x$. The general case follows immediately from this case by working with $B\cup\{x\}$. Suppose, by way of contradiction, that the above-mentioned particular case of the assertion is false. Then there exists some bounded subset $B$ of $X$ which contains $x$, and some positive number $r>0$, and also a strictly monotone subsequence $(\ell_k)_{k=0}^{\infty}$ of natural numbers, such that for all $k\in \N\cup\{0\}$ there exists some $y_{\ell_k}\in B$ such that $\|S_{\ell_k}(y_{\ell_k})\|\leq r$. 

Because $S_k$ is nonexpansive for each $k\in\N\cup\{0\}$, as a finite composition of nonexpansive mappings, we have $\|S_{\ell_k}(x)-S_{\ell_k}(y_{\ell_k})\|\leq \|x-y_{\ell_k}\|\leq \textnormal{diam}(B)<\infty$ for all $k\in\N\cup\{0\}$. The previous inequalities and the triangle inequality imply that 
\begin{equation*} 
\|S_{\ell_k}(x)\|\leq \|S_{\ell_k}(x)-S_{\ell_k}(y_{\ell_k})\|+\|S_{\ell_k}(y_{\ell_k})\|\leq \textnormal{diam}(B)+r<\infty 
\end{equation*}
for each $k\in\N\cup\{0\}$, and so $\sup\{\|S_{\ell_k}(x)\|: k\in\N\cup\{0\}\}<\infty$. On the other hand, since we assumed that $\lim_{k\to\infty}\|S_k(x)\|=\infty$ and since $\lim_{k\to\infty}\ell_k=\infty$, we have, in particular, $\sup\{\|S_{\ell_k}(x)\|\,|\, k\in\N\cup\{0\}\}=\infty$. We obtained a contradiction which shows that the initial assumption cannot hold, namely $\lim_{k\to\infty}\|S_{k}(\cdot)\|=\infty$ uniformly on bounded sets  which contain $x$, and, as explained earlier, on all nonempty and bounded sets. 

\item Obviously if $(S_k(x))_{k=0}^{\infty}$ is unbounded for all $x\in X$, then it is unbounded for some $x\in X$. On the other hand, suppose that $(S_k(x))_{k=0}^{\infty}$ is unbounded for some $x\in X$. Then given an arbitrary $y\in X$, Part \beqref{item:UnBoundedSome_x} above with the nonempty and bounded set $B:=\{y\}$ shows that $(S_k(y))_{k=0}^{\infty}$ is unbounded, as required. 

\item This part is an immediate consequence of Lemma \bref{lem:BruckreichCorollary1.4} and Part \beqref{item:UnBoundedSome_x} above with $T_i:=T$ for all $i\in\N\cup\{0\}$. 
\end{enumerate}
\end{proof}

\begin{remark}
Theorem \bref{thm:UniformInfty} is related to, but different from, \cite[Proposition 1]{BaranyGoodmanPollack1991inbook}. That proposition  says that given a point $x$ in the Euclidean space $\R^n$, $n\in\N$, a nonempty and possibly infinite family $I$ of affine subspaces in $\R^n$, and the orthogonal projections $T_i:\R^n\to\R^n$, $i\in I$ onto the members of $I$, the product sequence $(S_{k,c}(x))_{k=0}^{\infty}$ defined by $S_{k,c}(x):=T_{c(k)}T_{c(k-1)}\cdots T_{c(0)}(x)$ for every $k\in\N\cup\{0\}$ is bounded for each control $c:I\to\N\cup\{0\}$ if and only if all of the sequences $(S_{k,c}(x))_{k=0}^{\infty}$, $c:I\to\N\cup\{0\}$ are uniformly bounded (independently of the control $c$). 

Note that in Theorem \bref{thm:UniformInfty} the uniform boundedness is with respect to the variable $x$ (which varies inside a bounded subset) and not with respect to the control $c$ (which is fixed in Theorem \bref{thm:UniformInfty}). Anyway, a careful inspection of the proof of \cite[Proposition 1]{BaranyGoodmanPollack1991inbook} shows that the proof actually holds, essentially word for word, in any normed space instead of Euclidean spaces and with any family $(T_i)_{i\in I}$ of nonexpansive mappings instead of orthogonal projections onto affine subspaces. 
\end{remark}

\begin{expl}\label{ex:(C1C2)k(C_3)}
Suppose that $X$ is a real Hilbert space of dimension 2 or above (possibly infinite). Let $C_1$ and $C_2$ be arbitrary nonempty, closed, convex and disjoint subsets having the property that the distance $dist(C_1,C_2):=\inf\{\|x-y\|\,|\,\,x\in C_1, y\in C_2\}$ between them is not attained, that is, $dist(C_1,C_2)<\|x-y\|$ for all $x\in C_1$ and $y\in C_2$. Such sets must be unbounded (see, e.g., \cite[Facts 5.1(i)]{BauschkeBorwein1994jour} or, for a more general setting, \cite[Theorem 5.1(xv)]{ReemCensor2025prep-BAP} or \cite[Theorem 1.1.]{Stiles1965b-jour}). 

For instance, if $X$ is $n$-dimensional for some $2\leq n\in \N$, then we can take $C_1:=\{(x_1,x_2,x_3,\ldots,x_n)\,|\, x_2\geq e^{x_1}+1\}$ and $C_2:=\{(x_1,x_2,x_3,\ldots,x_n)\,|\, x_2\leq -1-e^{x_1}\}$, and if $X$ is infinite dimensional then we can either take the isometric embedding of $C_1$ and $C_2$ in an $n$-dimensional subspace of $X$ (where $2\leq n\in \N$ is arbitrary) or we can choose an infinite orthonormal sequence $(u_k)_{k=1}^{\infty}$ and let $C_1:=\{\sum_{k=1}^{\infty}\xi_k u_k\,|\,(\xi_k)_{k=1}^{\infty}\in\ell_2, \,\xi_2\geq e^{\xi_1}+1\}$ and $C_2:=\{\sum_{k=1}^{\infty}\xi_k u_k\,|\,(\xi_k)_{k=1}^{\infty}\in\ell_2, \,\xi_2\leq -e^{\xi_1}-1\}$. In all of these cases $dist(C_1,C_2)=2$ but the distance is not attained. 

For an example in which both $C_1$ and $C_2$ are infinite dimensional closed affine subspaces and $dist(C_1,C_2)>0$ is not attained, see \cite[Example 4.3]{BauschkeBorwein1994jour} (based on \cite[p. 312]{FranchettiLight1986jour}; see also \cite[p. 354]{KosmolZhou1991jour}); as follows from \cite[Proposition 5.16]{BauschkeBorwein1996jour}, in this case the family $(C_1,C_2)$ is not regular and hence not innately regular. 

Now let $C_3$ be an arbitrary nonempty, closed and convex subset of $X$, let $I:=\{1,2,3\}$, and consider the control $c:\N\cup\{0\}\to I$ defined in the following manner: $(12)(3)(12)^2(3)(12)^3(3)(12)^4(3)\ldots$. In other words, $c(0):=1$, $c(1):=2$, $c(2):=3$, $c(3):=1$, $c(4):=2$, $c(5):=1$, $c(6):=2$, $c(7):=3$, $c(8):=1$, $c(9):=2$, $c(10):=1$, $c(11):=2$, $c(12):=1$, $c(13):=2$, $c(14):=3$, and so on. Let $T_k:=P_{C_{c(k)}}$ for every $k\in\N\cup\{0\}$, and let $(x^k)_{k=0}^{\infty}$ be defined as $x^{k+1}:=T_{k}(x^k)$ for all $k\in\N\cup\{0\}$, where $x^0\in X$ is arbitrary. For an illustration, 
\begin{equation}\label{eq:x1x14}
\begin{array}{ll}
x^1=P_{C_1}(x^0), & x^2=P_{C_2}P_{C_1}(x^0), \\
x^3=P_{C_3}(P_{C_2}P_{C_1})(x^0), & x^4 =P_{C_1}P_{C_3}(P_{C_2}P_{C_1})(x^0),\\
x^5=P_{C_2}P_{C_1}P_{C_3}(P_{C_2}P_{C_1})(x^0), & x^6 =P_{C_1}P_{C_2}P_{C_1}P_{C_3}(P_{C_2}P_{C_1})(x^0),\\
x^7=(P_{C_2}P_{C_1})^2P_{C_3}(P_{C_2}P_{C_1})(x^0), & x^8=P_{C_3}(P_{C_2}P_{C_1})^2P_{C_3}(P_{C_2}P_{C_1})(x^0),\\
x^9=P_{C_1}P_{C_3}(P_{C_2}P_{C_1})^2P_{C_3}(P_{C_2}P_{C_1})(x^0), & x^{10}=(P_{C_2}P_{C_1})P_{C_3}(P_{C_2}P_{C_1})^2P_{C_3}(P_{C_2}P_{C_1})(x^0), \\
x^{11}=P_{C_1}(P_{C_2}P_{C_1})P_{C_3}(P_{C_2}P_{C_1})^2P_{C_3}(P_{C_2}P_{C_1})(x^0), &
x^{12}=(P_{C_2}P_{C_1})^2P_{C_3}(P_{C_2}P_{C_1})^2P_{C_3}(P_{C_2}P_{C_1})(x^0),\\
x^{13}=P_{C_1}(P_{C_2}P_{C_1})^2P_{C_3}(P_{C_2}P_{C_1})^2P_{C_3}(P_{C_2}P_{C_1})(x^0), & 
x^{14}=(P_{C_2}P_{C_1})^3P_{C_3}(P_{C_2}P_{C_1})^2P_{C_3}(P_{C_2}P_{C_1})(x^0).
\end{array}
\end{equation}
We claim that $(x^k)_{k=0}^{\infty}$ is unbounded. Indeed, suppose for the sake of contradiction that $(x^k)_{k=0}^{\infty}$ is bounded. Then there exists some $\alpha>0$ such that 
\begin{equation}\label{eq:bounded}
\|x^k\|\leq \alpha,\quad \forall k\in\N\cup\{0\}.
\end{equation} 
For each $\ell\in\N$, let $k_{\ell}\in \N$ satisfy 
\begin{equation}\label{eq:xkl}
x^{k_{\ell}}=Q^{\ell}x^{k_{\ell}-2\ell}, 
\end{equation}
where $Q:=P_{C_2}P_{C_1}$. For instance (see \beqref{eq:x1x14}), if $\ell:=1$, then $k_1=2$, if $\ell:=2$, then $k_{2}=7$, if $\ell:=3$, then $k_{\ell}=14$, and so on, by the definition of the control that we consider (and by induction $k_{\ell}> 2\ell$ for all $2\leq\ell\in\N$). Then \beqref{eq:bounded} implies that $\|x^{k_{\ell}}\|\leq\alpha$ for all $\ell\in\N$.

We claim that there exists some $\wt{k}\in\N$ sufficiently large  such that $\|Q^{k} y\|>\alpha$ for all $\wt{k}<k\in\N$ and all $y\in B[0,\alpha]$. Indeed, we first observe that the subset $Q(B[0,\alpha])$ is bounded (since it is the image of the bounded set $B[0,\alpha]$ by the nonexpansive mapping $Q$). 

Now we observe that $Q$ is strongly nonexpansive: indeed, $Q$ is the composition of two firmly nonexpansive mappings (the orthogonal projections $P_{C_2}$ and $P_{C_1}$), and hence $Q$ is a composition of two strongly nonexpansive mappings (since \cite[Proposition 2.1]{BruckReich1977jour} ensures that any firmly nonexpansive on a uniformly convex Banach space - and, therefore, on a Hilbert space - is strongly nonexpansive; see also \cite[Theorem 2.3.4, p. 92]{Cegielski2012book} for another generalization). Since the composition of strongly nonexpansive mappings is also strongly nonexpansive as follows from \cite[Proposition 1.1]{BruckReich1977jour} or \cite[Theorem 2.3.5(i), p. 93]{Cegielski2012book}, we conclude that indeed $Q$ is strongly nonexpansive. 

Next we observe that since $dist(C_1,C_2)$ is not attained, then \cite[Theorem 2]{CheneyGoldstein1959jour} ensures that $Q$ is fixed point free. Hence, since $Q$ is strongly nonexpansive, we can use Theorem \bref{thm:UniformInfty}\beqref{item:FixedPointFree} to obtain that $\lim_{q\to\infty} \|Q^q(\cdot)\|=\infty$ uniformly on nonempty and bounded sets of $X$. In particular, according to Theorem \bref{thm:UniformInfty}\beqref{item:FixedPointFree}, for the pair consisting of the nonempty and bounded set $Q(B[0,\alpha])$ and the positive number $\alpha$ there exists some $\wt{q}\in \N$ such that  $\|Q^q(z)\|>\alpha$ for all $\wt{q}<q\in\N$ and all $z\in Q(B[0,\alpha])$. As a result, if we   take an arbitrary $y\in B[0,\alpha]$, since the vector $z:=Q(y)$ is located in $Q(B[0,\alpha])$ it follows that $\|Q^q(z)\|>\alpha$ for all $\wt{q}<q\in\N$. Therefore if we let $\wt{k}:=\wt{q}+1$, then we see that $\|Q^{k}(y)\|>\alpha$ for all $\wt{k}<k\in\N$. 

Now let $\ell\in\N$ be large enough such that $\ell>\wt{k}$. If we denote  $y:=x^{k_{\ell}-2\ell}$, then $y\in B[0,\alpha]$ according to \beqref{eq:bounded}, and so by \beqref{eq:xkl}, the choices of $\wt{k}$, $k_{\ell}$ and the previous paragraph, we have $\|x^{k_{\ell}}\|=\|Q^{\ell}(y)\|>\alpha$, a contradiction to \beqref{eq:bounded}. This contradiction shows that $(x^k)_{k=0}^{\infty}$ cannot be bounded, as claimed.
\end{expl}

\begin{expl}\label{ex:Aleph}
As in Example \bref{ex:(C1C2)k(C_3)}, let $X$ be a real Hilbert space of dimension 2 or above (possibly infinite) and let $C_1$ and $C_2$ be any nonempty, disjoint, closed and convex subsets of $X$ having the property that the distance between them is not attained. Let $C_3,\ldots, C_m$ be arbitrary nonempty closed and convex subsets of $X$ and $3\leq m\in\N$. Let $I:=\{1,2,\ldots,m\}$. Let $(k_{\ell})_{\ell\in\N}$ be a strictly increasing sequence of natural numbers, and let $(j_{\ell})_{\ell\in\N}$ an arbitrary (not necessarily increasing) sequence of natural numbers. 

Consider an arbitrary control $c:\N\cup\{0\}\to I$ of the following form: 
\begin{equation*}
(12)^{k_1}(i_{1,1}i_{1,2}\ldots i_{1,j_1})(12)^{k_2}(i_{2,1},\ldots,i_{2,j_2})(12)^{k_3}(i_{3,1},\ldots,i_{3,j_3})\ldots,
\end{equation*}
where $i_{u,v}\in \{3,\ldots,m\}$ for all $u\in\N$, $v\in \{1,\ldots,j_u\}$. 

Let $T_k:=P_{C_{c(k)}}$ for every $k\in\N\cup\{0\}$, and let $(x^k)_{k=0}^{\infty}$ be defined as $x^{k+1}:=T_{k}(x^k)$ for all $k\in\N\cup\{0\}$, where $x^0\in X$ is arbitrary. By considerations similar (albeit longer) to the ones presented in Example \bref{ex:(C1C2)k(C_3)}, one can show that $(x^k)_{k=0}^{\infty}$ is unbounded.

What perhaps is noteworthy in this example is that any control of the above form is non-cyclic and not even almost cyclic, and hence whenever $C_3,\ldots, C_m$  are disjoint from $C_1$ and $C_2$ any sequence $(x^k)_{k=0}^{\infty}$ generated as above is non-cyclic and non-almost-cyclic. In addition, there are $\aleph$ controls of this form (for a given family $(C_i)_{i=1}^m$, $3\leq m\in\N$) since the cardinal number of the set $\mathcal{S}$ all strictly increasing sequences of natural numbers is $\aleph={\aleph_0}^{\aleph_0}$ (indeed, this known fact can be shown by using the following one-to-one bijection from the set $\N^{\N}$ of all sequences of natural numbers onto $\mathcal{S}$: $(s_k)_{k=1}^{\infty}\mapsto (\sum_{j=1}^k s_j)_{k=1}^{\infty}$ for each sequence $(s_k)_{k=1}^{\infty}\in \N^{\N}$). 
\end{expl}

\section*{Acknowledgments}
The authors thank Heinz Bauschke and Tran Thanh Tung for helpful discussions regarding \cite{BauschkeTung2026prep}. This work was supported by U.S. National Institutes of Health Grant Number R01CA266467 and by the Cooperation Program in Cancer Research of the German Cancer Research Center (DKFZ) and Israel’s Ministry of Innovation, Science and Technology (MOST).

\bibliographystyle{acm}
\bibliography{biblio}

\end{document}